%% file: RFK.tex
\documentclass[11pt,reqno]{amsart}
\include{Preamble}
\makeatother

\date{}
\begin{document}
\title[Reverse Faber-Krahn inequality in the Hyperbolic space]{ Reverse Faber-Krahn inequality for the $p$-Laplacian in Hyperbolic space}

\author{Mrityunjoy Ghosh$^{1,*}$}
\blfootnote{$^{*}$Corresponding author.}
\address{$^1$Department of Mathematics, Indian Institute of Technology Madras, Chennai 600036, India}

\author{Sheela Verma$^2$}
\address{$^2$Department of Mathematical Sciences \\
Indian Institute of Technology(BHU), Varanasi, India}
\email{ghoshmrityunjoy22@gmail.com, sheela.mat@iitbhu.ac.in}



\maketitle
\begin{abstract}
In this paper, we study the shape optimization problem for the first eigenvalue of the $p$-Laplace operator with the mixed Neumann-Dirichlet boundary conditions on multiply-connected domains in hyperbolic space. Precisely, we establish that among all multiply-connected domains of a given volume and prescribed $(n-1)$-th quermassintegral of the convex Dirichlet boundary (inner boundary), the concentric annular region produces the largest first eigenvalue. We also derive  Nagy's type inequality for outer parallel sets of a convex domain in the hyperbolic space. 
\end{abstract}

\input{Introduction}

\input{Preliminaries}
\input{Main_results}

\input{Open_problem}
\input{Acknowledgements}

\bibliographystyle{abbrvurl}
\bibliography{Reference}
\end{document}

%% file: Preamble.tex
\usepackage{setspace}
\usepackage{footnote}
\usepackage{graphicx}
\usepackage{cmap,mathtools}
\usepackage[T1]{fontenc}
\usepackage[utf8]{inputenc}
\usepackage[english]{babel}
\usepackage{amsfonts,amssymb,amsthm,amsmath,enumerate,bm,cite,mathrsfs}
\usepackage{url}
\usepackage[shortlabels]{enumitem}
\usepackage{float}
\usepackage{secdot}

\allowdisplaybreaks

\usepackage{graphicx}
\usepackage{epstopdf}
\usepackage{a4wide}
\usepackage{xparse}
\usepackage[dvipsnames]{xcolor}
\usepackage[normalem]{ulem}

\usepackage[colorlinks=true,linktocpage,pdfpagelabels,bookmarksnumbered,bookmarksopen]{hyperref}
\definecolor{ForestGreen}{rgb}{0.1,0.8,0}
\definecolor{EgyptBlue}{rgb}{0.063,0.1,0.6}
\hypersetup{
	colorlinks=true,
	linkcolor=blue,         
	citecolor=ForestGreen,
	urlcolor=purple
}

\usepackage[hyperpageref]{backref}

\usepackage[margin=1in]{geometry}

\newcommand\blfootnote[1]{%
  \begingroup
  \renewcommand\thefootnote{}\footnote{#1}%
  \addtocounter{footnote}{-1}%
  \endgroup
}

\newtheorem{theorem}{Theorem}
\newtheorem{proposition}[theorem]{Proposition}
\newtheorem{lemma}[theorem]{Lemma}

\theoremstyle{definition}
\newtheorem{definition}[theorem]{Definition}
\newtheorem{remark}[theorem]{Remark}

\let\OLDthebibliography\thebibliography
\renewcommand\thebibliography[1]{
	\OLDthebibliography{#1}
	\setlength{\parskip}{1pt}
	\setlength{\itemsep}{1pt plus 0.3ex}
}

\numberwithin{theorem}{section}
\numberwithin{theorem}{section}

\DeclarePairedDelimiter\norm{\lVert}{\rVert}%
\makeatletter
\let\oldnorm\norm
\def\norm{\@ifstar{\oldnorm}{\oldnorm*}}

\newcommand{\pa} {\partial}
\newcommand{\be} {\beta}
\newcommand{\de} {\delta}
\newcommand{\De} {\Delta}

\newcommand{\ga} {\gamma}
\newcommand{\Ga} {\Gamma}
\newcommand{\om} {\omega}
\newcommand{\Om} {\Omega}

\newcommand{\ka} {\kappa}

\newcommand\restr[2]{{
  \left.\kern-\nulldelimiterspace 
  #1 
  \right|_{#2} 
  }}

\def\vol{\mathrm{Vol}}
\def\B{{\mathcal B}}

\def\L{{\widetilde L}}
\def\M{{\widetilde M}}

\def\t{\tau}

\def\C{{\mathcal C}}

\def\R{{\mathbb R}}

\def\Hn{{\mathbb H}^n}
\def\H{{\mathbb H}}

\def\({{\Big(}}
\def\){{\Big)}}

\def\c1{{\C_c^1}}

\def\d{{\rm d}}
\def\dr{{\rm d}r}
\def\dS{{\rm d}S}

\def\dt{{\rm d}t}

\def\G{{\widetilde{G}}}

\usepackage[foot]{amsaddr}
\usepackage[pagewise]{lineno}

%% file: Introduction.tex
\section{Introduction}\label{intro}

The study of isoperimetric type inequalities for the eigenvalues of elliptic operators remains one of the most attracted areas in spectral theory after a famous conjecture by Lord Rayleigh stating that:
\emph{among all domains of the given volume, the ball minimizes the first eigenvalue $\lambda_{1}$ of the Dirichlet Laplacian}, i.e.,
\begin{align} \label{RFK inequality}
 \lambda_{1}(\Omega) \geq \lambda_{1} (B),  
\end{align}
for all domains $\Omega$ such that Vol$(\Omega)$= Vol$(B)$. Here $B$ represents the ball. This conjecture was proved by Faber \cite{Faber} for planar Euclidean domains, and later Krahn \cite{Krahn1926} generalized it to higher dimensions. Inequality \eqref{RFK inequality} is known as the \textit{Rayleigh-Faber-Krahn inequality}. Similar results also hold for domains in Riemannian manifolds; see \cite{Chavel,Benguria} for instance. We refer to the monographs \cite{Henrot, Henrot2021} for various such isoperimetric type problems.

 In this article, we focus on the first eigenvalue of the $p$-Laplace operator with the mixed Neumann-Dirichlet boundary conditions on domains in the hyperbolic space. Let $\Hn$ denote the $n$-dimensional hyperbolic space with constant sectional curvature $-1$. Let $\Om\subset \Hn$ be a bounded domain with $\pa \Omega=\Ga_D\sqcup \Ga_N$.  For $1<p<\infty,$ the $p$-Laplace operator is defined as $\De_p u=\text{div}(|\nabla u|^{p-2} \nabla u)$. Here $\nabla$ denotes the hyperbolic gradient. For $p=2,$ the $p$-Laplace operator coincides with the classical \textit{Laplace-Beltrami} operator. We consider the following eigenvalue problem of the $p$-Laplace operator:
\begin{equation}\tag{$\mathscr{P}$}\label{Problem}
	\left. \begin{aligned}
	-\Delta_p u &=\tau |u|^{p-2}u \quad\text{in} \quad\Omega,\\
	u &=0 \qquad\quad \quad\;\text{on}\quad \Ga_D,\\
\frac{\partial u}{\partial \eta} &=0 \qquad\quad \quad\;\text{on} \quad \Ga_N,
	\end{aligned}\right\}
\end{equation}
where $\tau\in \R$ and $\eta$ is the outward unit normal vector to $\Ga_N$. A real number $\tau$ is said to be an eigenvalue of \eqref{Problem} if there exists $\phi\in W^{1,p}_{\Ga_D}(\Omega)\setminus\{0\}$ satisfying the following
\begin{equation*}
	\int_{\Om}|\nabla \phi|^{p-2}\left<\nabla \phi,\nabla w \right> \d V_g=\tau\int_{\Om}|\phi|^{p-2}\phi w\; \d V_g, \quad\forall\; w\in W^{1,p}_{\Ga_D}(\Omega),
\end{equation*}
where $\d V_g$ is the volume element induced by the hyperbolic metric $g$ and $W^{1,p}_{\Ga_D}(\Omega)$ is the  space of all Sobolev functions that vanishes on $\Ga_D$, i.e., 
\begin{equation*}
    W^{1,p}_{\Ga_D}(\Omega)=\{u\in W^{1,p}(\Om):u|_{\Ga_D}=0 \}.
\end{equation*}
It is well known that \eqref{Problem} admits a least positive eigenvalue $\t_1(\Om)$ (cf. \cite{Azorero1987}) whose variational characterization is given by 
\begin{equation}\label{Variational}
    \t_1(\Om)=\inf_{u\in W^{1,p}_{\Ga_D}(\Omega)\setminus\{0\}}\left\{\frac{\int_{\Omega}|\nabla u|^p\d V_g}{\int_{\Omega}|u|^p\d V_g}\right\}
\end{equation}
and $\tau_1(\Omega)$ is simple. 

Let $W_{n-1}(C)$ denotes the $(n-1)$-th quermassintegral (see Section \ref{Quermassin} for precise definition) of a convex domain $C$. In this article, we choose the following types of domains:
\begin{equation}\tag{$\mathscr{D}$}\label{Domain}
	\left. \begin{aligned}
	\Om =\Om_N\setminus \overline{\Om_D},\;\text{where}\; \Om_D, \Om_N \;\text{are two smooth, bounded domains in}\\
	 \H^n \;
	\text{ such that}\;\Om_D\;\text{is simply connected and} \;\overline{\Om_D}\subset \Om_N. \\
	A_\Om = B_R\setminus \overline{B_r},\;\text{where}\; B_R, B_r \;\text{are two concentric open geodesic balls}\\
	\text{ of radius}\;R, r\;(0<r<R)\;\text{respectively in}\;\Hn \;\text{such that}\\
	|\Om|=|A_\Om|\;\text{and}\;W_{n-1}(B_r)=W_{n-1}(\Om_D).
\end{aligned}\right\}
\end{equation}
 Assume that $\Ga_D := \partial \Om_D$ and $\Ga_N := \partial \Om_N$. Here $\Ga_D$ and $\Ga_N$, respectively, represent the Dirichlet and Neumann boundary, i.e., we consider the inner Dirichlet-outer Neumann boundary condition for \eqref{Problem}.

Now we state some existing isoperimetric bounds of $\t_1(\Om)$ for domains in the Euclidean space. Suppose $\Om$ and $A_\Om$ are domains in $\R^n$ as defined in \eqref{Domain}. For $\Om \subset \R^2$, Hersch \cite{Hersch} studied problem \eqref{Problem} for the classical Laplace operator and proved that $A_\Om$ maximizes the first eigenvalue of \eqref{Problem}, i.e., 
$$\t_1(\Om)\leq \t_1(A_\Om).$$
The above inequality is known as the \textit{reverse Faber-Krahn} inequality for the mixed eigenvalue problem. Note that in the planar case, the quermassintegral constraint, imposed on the Dirichlet boundary, reduces to the perimeter constraint (see Section \ref{Quermassin}). In \cite[Theorem 1.2]{AnoopAshok}, Anoop and Ashok extended Hersch's result for the $p$-Laplacian and to the higher dimensions under the assumptions that $\Om_D$ is a ball. Later, in \cite[Theorem 1.1]{Dellapietra}, the authors extended this result to the case when $\Om_D$ is convex. The proof given by Hersch \cite{Hersch} is based on the \textit{``method of interior parallels''} for planar domains. Hersch's idea was to construct a test function whose level sets are the parallel sets to the Dirichlet boundary. The key step for applying this method is the Nagy's inequality \cite{Nagy} for outer parallel sets of a planar domain, which is as follows: \\
\emph{Let $K\subset \R^2$ be a bounded, simply connected domain and $\de>0$. Let $K_\de$ denotes the set of all points in $\R^2$ that are at a distance (Euclidean) at most $\de$ from $K$. Suppose $K^\#$ is an open ball in $\R^2$ of same perimeter as $K$, i.e., $P(K)=P(K^\#).$ Then Sz. Nagy \cite{Nagy} proved that}
\begin{equation}\label{Nagy_Rn}
    P(K_\de)\leq P(K^\#_\de).
\end{equation}
In \cite{AnoopAshok}, the authors derive an analogoue of the above inequality for multiply connected domains in higher dimensions under the assumption that $\Om_D$ is a ball. However, a rigorous version of Nagy's type inequality \eqref{Nagy_Rn} for convex domains in $\R^n$ $(n\geq 3)$ has been proved in \cite[Corollary 3.4]{Joy}. It is worth mentioning that Nagy's type inequality has its own importance as it can be applied to obtain several bounds for the first eigenvalue of Laplacian and torsional rigidity; see \cite{Makai1,Polya}, for instance. To the best of our knowledge, the analogue of Hersch's result and Nagy's type inequality for the outer parallel sets are not available in hyperbolic space.

The main objective of this article is to prove Hersch's result in the hyperbolic space $\Hn$. Moreover, we establish a hyperbolic version of Nagy's inequality \eqref{Nagy_Rn}. To state the main results, we need the following definitions.

\begin{definition}[Outer parallel set] \label{Parallel}
Let $K\subset \Hn$ and $\de>0$. Then the \textit{Outer parallel body} of $K$  at a distance $\de>0$ is defined as 
$$K_\de=\{x\in \Hn: d_{\H} (x,K)\leq \de\},$$
where $d_\H$ is the hyperbolic distance function. The boundary $\pa K_\de$ is called as the \textit{Outer parallel set} of $K$ at a distance $\de.$
\end{definition}

Next, we recall the definition of \textit{h-convex} (or \textit{horoconvex}) domains in the hyperbolic space; cf. \cite[Section 2]{Solanes2004}. For more details, see Section \ref{Sec:Horoconvexity}.
\begin{definition}[$h$-convex domain]\label{Hconvex}
A domain $K\subset \Hn$ is said to be \textit{h-convex} if all the principal curvatures of $\pa K$ are bounded below by 1, i.e., if $\ka_i,\;1\leq i\leq  n-1,$ are the principal curvatures of $\pa K$, then $\ka_i\geq 1,\;\forall \;1\leq i\leq  n-1$.
\end{definition}

Let $P(A):=|\pa A|$ denotes the perimeter of a set $A\subset \Hn.$ Now we state the first main result of this article which is an analogue of Sz. Nagy's inequality for outer parallel sets of a domain in the hyperbolic space. 
\begin{theorem}[Nagy's inequality]\label{Nagy_theorem}
Let $K\subset \Hn$ be a smooth, bounded, convex domain and $\de>0$. Let $K^*$ be an open geodesic ball in $\Hn$ such that $W_{n-1}(K)=W_{n-1}(K^*)$. Then the followings hold:
\begin{enumerate}[(i)]
    \item If $n=2$, then $P(K_\de)\leq P(K^*_\de).$ Further, equality holds if and only if $K$ is a geodesic ball.
    \item  If $n\geq 3$ and $K$ is h-convex, then $P(K_\de)\leq P(K^*_\de).$ Further, equality holds if and only if $K$ is a geodesic ball.
\end{enumerate}
\end{theorem}

The main ingredients to prove Theorem \ref{Nagy_theorem} are $(i)$ the Steiner formula for convex domains in $\Hn$, and $(ii)$ classical hyperbolic isoperimetric inequality (for $n=2$) and a version of Alexandrov-Fenchel inequality involving the quermassintegrals due to Wang and Xia \cite{Wang2014} (for $n\geq 3$). First, we express the perimeter of outer parallel sets of a convex domain in terms of a polynomial in $\de$ using the Steiner formula. Then we derive an  isoperimetric type inequality between $W_i(K)$ and $W_i(K^*)$, which gives the desired result upon substituting in the Steiner formula. At this point, it is necessary to mention that for $n=2$, we are able to get Nagy's type estimate for the convex domains, thanks to the classical hyperbolic isoperimetric  inequality that holds for any domain. However, for $n\geq3$, we need a stronger assumption than convexity, called $h$-convexity. This assumption is necessary to apply a class of Alexandrov-Fenchel inequalities (Proposition \ref{Alexandrov}) which are not available for convex domains in the hyperbolic space. The extension of Theorem \ref{Nagy_theorem} for general domains in the hyperbolic space seems a challenging open problem.

Then by applying the Nagy's inequality (Theorem \ref{Nagy_theorem}), we prove the reverse Faber-Krahn inequality for domains in the hyperbolic space. More precisely, we obtain the following result.

\begin{theorem}[Reverse Faber-Krahn inequality]\label{RFK_theorem}
Let $\Om, A_\Om$ be as defined in \eqref{Domain} and $\t_1$ be the first eigenvalue of \eqref{Problem}. Assume that $\Om_D$ is convex for $n=2$ and $\Om_D$ is h-convex for $n\geq 3.$ Then
 $$\t_1(\Om)\leq \t_1(A_\Om).$$ 
 Moreover, equality occurs only when $\Om=A_\Om.$
\end{theorem}

To prove Theorem \ref{RFK_theorem}, we apply the method of interior parallels in the hyperbolic space with the help of Nagy's type inequality (Theorem \ref{Nagy_theorem}) for outer parallel sets. Namely, we produce a test function on $\Om$ using the first eigenfunction of $A_\Om$ that remains constant on the outer parallel sets to $\Om_D$. Indeed, the construction of the test function on $\Om$ is done in such a way that its gradient norm coincides with the first eigenfunction of $A_\Om$, whereas its $p$-norm increases. We would like to mention that the analogue of Theorem \ref{RFK_theorem} for the case when $\Om_D$ is a non-convex domain remains completely open (see Section \ref{open_problem}). 

The rest of this article is organized as follows. In Section \ref{preli}, we discuss a few geometric tools related  to the convex domains in the hyperbolic space and mention some facts about the $h$-convexity. The proofs of Theorem \ref{Nagy_theorem} and Theorem \ref{RFK_theorem} are given in Section \ref{main_results}. Finally, in Section \ref{open_problem}, we mention some open problems related to Nagy's type inequality and reverse Faber-Krahn inequality.

%% file: Preliminaries.tex
\section{Preliminaries}\label{preli}
In this section, we first discuss the notion of quermassintegrals (or mixed volumes) for a convex domain in $\Hn$. Then we state a few well known facts about the $h$-convex domains. We complete this section by providing some isoperimetric inequalities in the hyperbolic space, which will be used in later sections. Throughout the article, we denote the boundary of a set $A\subset \Hn$ by $\pa A.$ Also, $P(A)$ stands for the perimeter of $ A$, i.e., $P(A)=|\pa A|.$

\subsection{Quermassintegrals \& Curvature integrals}\label{Quermassin} Let $K\subset \Hn$ be a bounded, convex domain. Then the $Quermassintegrals$ $W_j(K),\;\text{for}\;1\leq j\leq n-1,$ is defined (cf. \cite{Wang2014,Santalo}) as
\begin{equation}\label{Quermass}
    W_j(K)=\frac{(n-j)\om_{j-1}\cdots\om_0}{n\om_{n-2}\cdots\om_{n-j-1}}\int_{\mathcal{L}_j}\chi(L_j\cap K)\d L_j,
\end{equation}
where $L_j$ is a $j$-dimensional totally geodesic subspace, $\mathcal{L}_j$ is the space of all totally geodesic subspaces of dimension $j$, $\d L_j$ is the natural measure on $\mathcal{L}_j$, $\om_i$ denotes the $i$-dimensional Hausdorff measure of the $i$-dimensional unit sphere and $\chi$ is the characteristic function acting as $\chi(A)=1,$ if $A\neq \emptyset$ and $\chi(A)=0,$ if $A= \emptyset$. As a convention, we assume $W_0(K)=\vol(K)$ and $W_n(K)=\frac{\om_{n-1}}{n}.$ Also we observe that $W_1(K)=\frac{P(K)}{n}$; cf. \cite{Santalo}.

Let $\ka_1,\ka_2,\dots,\ka_{n-1}$ are the principal curvatures of $\pa K$ and $H_j$, for $0\leq j\leq n-1$, denote the normalized elementary symmetric functions of principal curvatures of $\pa K$. Then the \textit{Curvature integrals} are defined by
\begin{equation}\label{Curvature_int}
    V_{n-j-1}(K)=\int_{\pa K} H_j\dS, \;\text{for}\;j=0,1,\dots,n-1,
\end{equation}
  where $\dS$ is the volume element on $\pa K$ induced from $\Hn$. Now by \cite[Proposition 7]{Solanes2005}, curvature integrals and quermassintegrals are related by the following formula:
\begin{equation}\label{Curv_Quer}
    V_{n-j-1}(K)=n\left(W_{j+1}(K)+\frac{j}{n-j+1}W_{j-1}(K)\right), \;\text{for }\;0\leq j\leq n-1.
\end{equation}
  
 \subsection{Horoconvexity ($h$-convexity) in the hyperbolic space} \label{Sec:Horoconvexity} We first define $h$-convexity in the hyperbolic plane via $\lambda$-geodesics and state some of its properties. Then we give the definition of $h$-convexity in higher dimensions. For more details, see \cite{gallego1999asymptotic}.

\begin{definition} [Equidistants or $\lambda$- geodesics]
The curves which are equidistant to geodesics are called \emph{Equidistants}. A $\lambda$-geodesic is an equidistant that meets the infinity line with angle $\alpha$ such that $| \cos \alpha | = \lambda$.
\end{definition}

\begin{remark}
For $\lambda = 0  (\alpha = 90^{\circ})$, equidistants are geodesics and for $\lambda = 1$, they are horocycles. The geodesic curvature of a $\lambda$ geodesic is $\pm \lambda$. Some equidistants in the hyperbolic plane have been drawn in Figure \ref{fig:Equidistants}.
\end{remark}

\begin{figure}[htp]
    \centering
    \includegraphics[width=10cm]{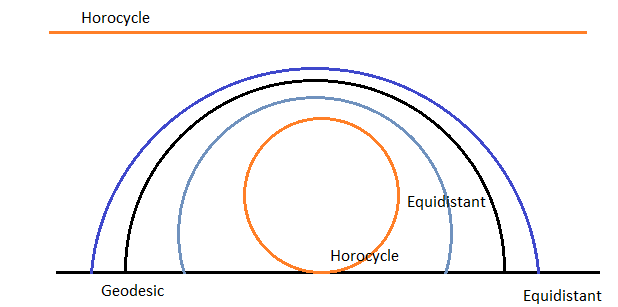}
    \caption{Equidistants in the hyperbolic plane}
    \label{fig:Equidistants}
\end{figure}
The following lemma shows the relation between positions of different $\lambda$ geodesics.
\begin{lemma}
Given any two points $p$ and $q$ in the hyperbolic plane and $0 < \lambda \leq 1$, there are exactly two $\lambda$-geodesics passing through them. These $\lambda$-geodesics are symmetric with respect to the geodesic passing through $p$ and $q$ and lie in the region bounded by the two horocycles passing through these points. 
\end{lemma}
Now we define $\lambda$-convexity of a set in the hyperbolic plane, $h$-convexity is the particular case of this.
\begin{definition}
For given $\lambda \in [0,1]$, a set $\Omega$ in the hyperbolic plane is said to be $\lambda$-convex if for every $p,q \in \Omega$, the $\lambda$-geodesics joining them lie inside $\Omega$. $1$-convex sets are also called $h$-convex sets.
\end{definition}

\begin{lemma}
A compact domain $\Omega$ with $C^{2}$-boundary is $\lambda$-convex if and only if the geodesic curvature $k_{g}$ of the boundary satisfies $k_{g} \geq \lambda$ ($k_{g} \leq -\lambda$, in case of opposite orientation).
\end{lemma}

\begin{remark}\label{H_convex}
If a domain is $\lambda_{0}$-convex then it is $\lambda$-convex for all $\lambda \leq \lambda_{0}$. In particular, every $h$-convex set is convex but converse is not true. For example, consider convex polygon.

In higher dimensions, $h$-convexity can be defined in the similar way. We define $h$-convexity of domains in $n$-dimensional hyperbolic space $\Hn$ in terms of horospheres.
\end{remark}

 A horoball is the limit of a sequence of increasing balls sharing a tangent hyperplane and its point of tangency. A horosphere is the boundary of a horoball.

\begin{definition}
A domain $\Omega \subset \Hn$ is said to be horoconvex (or $h$-convex) if, for every point $p \in \partial \Omega$, there exists a horosphere passing through the point $p$ such that the domain $\Omega$ lies entirely in the horoball bounded by the horosphere.
\end{definition}


\subsection{Steiner formula \& Alexandrov–Fenchel inequality} Let $K\subset \Hn$ be a smooth, bounded, convex domain and $\de>0.$ Then by Steiner formula (cf. \cite[Chapter 18, Section 4]{Santalo}), the volume of $K_\de$ is given by
\begin{equation}\label{Steiner_formula}
    \vol(K_\de):=|K_\de|=\vol(K)+\sum_{j=0}^{n-1}\binom{n}{j}V_j(K)\int_{0}^{\de}\cosh^j(t)\sinh^{n-j-1}(t)\dt.
\end{equation}
Therefore, the perimeter $P(K_\de):=|\pa K_\de|$ of $K_\de$ has the following expansion:
\begin{equation}\label{Steiner_1}
    P(K_\de)=\frac{d}{d\de}(\vol(K_\de))=\sum_{j=0}^{n-1}\binom{n}{j}V_j(K)\left(\frac{d}{d\de}\int_{0}^{\de}\cosh^j(t)\sinh^{n-j-1}(t)\dt\right).
\end{equation}
Now using \eqref{Curv_Quer}, we restate \eqref{Steiner_1} in terms of the quermassintegrals of $K$:
\begin{align}\label{Steiner}
    P(K_\de)=\sum_{j=0}^{n-2}n\binom{n}{j}\Big\{ W_{n-j}(K)&+\frac{n-j-1}{j+2}W_{n-j-2}(K)\Big\}\left(\frac{d}{d\de}\int_{0}^{\de}\cosh^j(t)\sinh^{n-j-1}(t)\dt\right)\nonumber\\&+n^2W_1(K)\left(\frac{d}{d\de}\int_{0}^{\de}\cosh^{n-1}(t)\dt\right).
\end{align}
{\bf Hyperbolic Isoperimetric inequality:} Let $\ga$ be a closed curve in $\H^2$ and $K$ be the domain enclosed by $\ga$. Then the hyperbolic isoperimetric inequality (cf. \cite{Schmidt}) states that 
\begin{equation}\label{Isoperimetric}
    P(K)^2\geq 4\pi |K|+ |K|^2,
\end{equation}
where $|K|$ is the area of $K$ and $P(K)$ is the length of $\ga$. Furthermore, equality occurs if and only if $\ga$ is a circle.

Next, we state an isoperimetric inequality between the quermassintegrals of an h-convex domain in hyperbolic space obtained by Wang and Xia \cite[Theorem 1.1]{Wang2014}.
\begin{proposition}[Alexandrov–Fenchel inequality in $\Hn$]\label{Alexandrov}
Let $K\subset \Hn$ be a smooth, h-convex, bounded domain. Also let $f_m:[0,\infty)\longmapsto \R_+$ be defined by $f_m(r)=W_m(B_r),$ where $B_r$ is the geodesic ball of radius $r$. Then for $0\leq i<j\leq n-1$,
$$W_j(K)\geq f_j\circ f_i^{-1}(W_i(K)).$$
The equality occurs if and only if $K$ is a geodesic ball.
\end{proposition}

%% file: Main_results.tex
\section{Main results}\label{main_results}
In this section, we first give a proof of Theorem \ref{Nagy_theorem}.
\begin{proof}[Proof of Theorem \ref{Nagy_theorem}]
$(i)$ Since $n=2,$ we have $P(K)=P(K^*)$, i.e., $W_1(K)=W_1(K^*)$. Therefore, by isoperimetric inequality \eqref{Isoperimetric}, we have $|K|\leq |K^*|$, i.e., $W_0(K)\leq W_0(K^*).$ Hence from Steiner formula \eqref{Steiner}, we get
\begin{align}
    P(K_\de)&\leq 2\Big\{W_{2}(K^*)+\frac{1}{2}W_0(K^*)\Big\}\left(\frac{d}{d\de}\int_{0}^{\de}\sinh(t)\dt\right)+4W_1(K^*)\left(\frac{d}{d\de}\int_{0}^{\de}\cosh(t)\dt\right)\nonumber\\
    &=P(K^*_\de).\nonumber
\end{align}
The equality case follows immediately from the isoperimetric inequality \eqref{Isoperimetric}.

$(ii)$ Given that 
\begin{equation}\label{Assumption}
    W_{n-1}(K)=W_{n-1}(K^*).
\end{equation}
Let $0\leq j < n-1$. Since $K$ is $h$-convex, applying Proposition \ref{Alexandrov} for $j$ and $n-1$ and using \eqref{Assumption}, we get
\begin{align*}
    f_{n-1}\circ f_j^{-1}(W_j(K)) \leq W_{n-1}(K)&=W_{n-1}(K^*)=f_{n-1}\circ f_j^{-1}(W_j(K^*)).
    \end{align*}
    Thus
    \begin{equation}\label{Compar}
    f_{n-1}\circ f_j^{-1}(W_j(K)) \leq f_{n-1}\circ f_j^{-1}(W_j(K^*)).
\end{equation}
Now from \eqref{Quermass}, we observe that if $r_1<r_2$, then $W_{i}(B_{r_1})<W_{i}(B_{r_2}),$ for all $0\leq i\leq n-1$. Thus the function $r\longmapsto f_i(r)$ is a strictly increasing function for all $0\leq i\leq n-1$. Therefore inequality \eqref{Compar} immediately gives that
\begin{equation}\label{Quer_comp}
    W_j(K)\leq W_j(K^*),\;\text{for all}\; 0\leq j<n-1.
\end{equation}
Also the equality occurs in \eqref{Quer_comp} only when $K$ is a geodesic ball (by Proposition \ref{Alexandrov}). Thus using \eqref{Quer_comp} in \eqref{Steiner}, we have 
\begin{align}
P(K_\de)\leq \sum_{j=0}^{n-2}n\binom{n}{j}\Big\{ W_{n-j}(K^*)&+\frac{n-j-1}{j+2}W_{n-j-2}(K^*)\Big\}\left(\frac{d}{d\de}\int_{0}^{\de}\cosh^j(t)\sinh^{n-j-1}(t)\dt\right)\nonumber\\
&+n^2W_1(K^*)\left(\frac{d}{d\de}\int_{0}^{\de}\cosh^{n-1}(t)\dt\right)\nonumber\\
&= P(K_\de^*).\nonumber
\end{align}
Hence $P(K_\de)\leq P(K_\de^*).$ Now the equality case follows from \eqref{Quer_comp}. This completes the proof.
\end{proof}
\begin{remark}\label{Perimeter constraint} Note that for $n=2$, $W_{n-1}(K)=W_{n-1}(K^*)$ implies that $P(K)=P(K^*)$ (see Section \ref{Quermassin}). However, if $n\geq 3,$ then $P(K)<P(K^*)$. 
\end{remark}

Next, we prove an auxiliary result which is needed to prove our main result. Let us start with a notational set up. Let $\Om$ and $A_\Om$ be as stated in \eqref{Domain}, i.e., $\Om=\Om_N\setminus \overline{\Om_D}$ and $A_\Om=B_R\setminus \overline{B_r}$. Also, we have $W_{n-1}(\Om_D)=W_{n-1}(B_r),$ i.e., $\Om_D^*=B_r.$ Define
$$\B(\de)=\pa\Om_{D_\de}\cap \Om,\; L(\de)=|\B(\de)|,\;\text{for}\;\de>0,$$
$$\delta_0=\sup\{\de>0: \B(\de) \neq\emptyset \}.$$
Note that $L(\de)\leq P(\Om_{D_\de}),\;\forall\;\de>0.$ We set $\L(\delta):= P(B_{r_\de})$ for simplicity. For $p\in (1,\infty)$, we construct $M$ and $\M$ as follows:
	\begin{equation}\label{Parame}
		M(\delta)=\displaystyle\int_{0}^{\delta}\frac{1}{L(r)^{p'-1}}\dr ,\quad 
		\M(\delta)=\int_{0}^{\delta}\frac{1}{\L(r)^{p'-1}}\dr,
	\end{equation}
	where $p'=\frac{p}{p-1}$ is the holder conjugate of $p$.

\begin{remark} \label{Property_M}
\begin{enumerate}[(i)]
    \item Note that, both $\de\longmapsto M(\de)$ and $\de\longmapsto \M(\de)$ are strictly increasing functions on $[0,\de_0]$ and $[0, R-r]$, respectively. Moreover, $M(\de_0)$ can be infinite also.
    
    \item Observe that, by Theorem \ref{Nagy_theorem}, $L(\de)\leq \L(\de)$ for all $\de >0.$ Therefore, by the definitions of $M$ and $\M$, we immediately have $\M(\de)\leq M(\de)$ for all $\de>0.$

\end{enumerate}
\end{remark}
 
 The euclidean version of the following lemma has been proved in \cite[Lemma 2.7]{AnoopAshok} when $\Om_D\subset \R^n$ is a ball, and in \cite[Lemma 5.2]{Joy} when $\Om_D\subset \R^n$ is convex. We generalize these results for $\Om_D\subset \H^n$ using the hyperbolic analogue of Nagy's inequality.
 
\begin{lemma}\label{Auxiliary_lemma}
Suppose $\Om,A_\Om, M$ and $\M$ are as mentioned above. Assume that $\Om_D$ is convex for $n=2$ and $\Om_D$ is h-convex for $n\geq 3$. Then the followings hold:
\begin{enumerate}[(i)]
    \item $R-r\leq \de_0,$ with equality occurs only when $\Om$ is a concentric annular region.
    
    \item Define
	\begin{align*}\label{paramap}
	   & G(\be)=L(M^{-1}(\be)),\;\text{for}\;\be \in [0,M(\de_0)]\; \\
	   \text{and} \qquad & \G(\be)=\L(\M^{-1}(\be)),\;\text{for}\;\be \in [0,\M(R-r)].
	\end{align*}
	Then $G(\be)\leq \G(\be)$, for all $\be \in [0,\M(R-r)]$ and equality holds if and only if $\Om$ is a concentric annulus. Furthermore, if $\Omega$ is not a concentric annulus, then $G(\be)<\G(\be)$ on $[\be',\M(R-r)]$, for some $\be' \in [0,\M(R-r)]$.
\end{enumerate}
\end{lemma}

\begin{proof}
$(i)$ If possible, let $R-r>\de_0$. Now by Theorem \ref{Nagy_theorem}, we have $L(\de)\leq \L(\de)$. Therefore,
\begin{align*}
|A_\Om|=\int_{0}^{R-r}\L(\delta)\d\de&\geq	\int_{0}^{\de_0}L(\delta)\d\de+\int_{\de_0}^{R-r}\L(\delta)\d\de=|\Omega|+\int_{\de_0}^{R-r}\L(\delta)\d\de>|\Omega|,
	\end{align*}
which is a contradiction as $|\Om|=|A_\Om|$ (by assumption). Hence $ R-r\leq \de_0.$ Now if $\de_0= R-r$, then $$\displaystyle\int_{0}^{R-r}\L(\delta)\d\de=\int_{0}^{\de_0}L(\delta)\d\de \implies \displaystyle\int_{0}^{R-r}(\L(\delta)-L(\delta))\d\de=0.$$ 
Observe that both $L$ and $\L$ are continuous function and hence $L(\delta)=\L(\delta),$ for all $\delta\in[0,R-r]$. Further, by applying Theorem \ref{Nagy_theorem} for $\Om_D$, we get 
$$L(\delta)\leq P(\Om_{D_{\de}})\leq P(B_{r_{\de}}) =\L(\delta),\;\text{ for all}\;\delta\in[0,R-r].$$ 

Thus we have $P(\Om_{D_{\de}})= P(B_{r_{\de}})$. Therefore, by Theorem \ref{Nagy_theorem}, it follows that $\Om_D$ must be a geodesic ball. Also since $\de_0= R-r$, $\Ga_N$ has to be a geodesic sphere. Hence $\Om$ must be a concentric annulus.

$(ii)$ Let $M_*=M(\de_0)$ and $\M_*=\M(R-r)$. Then using $(i)$ and Remark \ref{Property_M}, we have $\M_*\leq M_*$. Therefore, $$M^{-1}(\be)\leq \M^{-1}(\be),\;\text{ for all }\;\be\in [0,\M_*].$$
Since $\de\longmapsto \L(\de)$ is an strictly increasing function on $[0, R-r]$, we get
$$G(\be)=L(M^{-1}(\be))\leq \L(M^{-1}(\be))\leq \L(\M^{-1}(\be))=\G(\be).$$
	Moreover, if $G(\be)=\G(\be),$ then $L(\de)=\L(\de),\;\text{ for all }\;\de\in [0,R-r]$. Thus the equality case follows immediately from $(i)$. Now if $\Om$ is not a concentric annulus, then by $(i)$, there exists $\delta'\in [0,R-r] $ such that  $L(\delta')<\L(\delta')$. Thus $\M(\delta')<M(\delta')$ and hence $\M(\delta)<M(\delta)$ for all $\delta'\leq\delta\leq R-r$. Now substituting $\be'=M(\delta')$ gives the desired conclusion.
\end{proof}

Now we state few properties of a first eigenfunction of \eqref{Problem} associated to $\t_1$.
\begin{proposition}\label{Properties}
Let $\Om$ be as mentioned in \eqref{Domain} and $\t_1(\Om)$ be the first eigenvalue of \eqref{Problem} on $\Om$. Suppose that $u$ is an  eigenfunction associated to $\t_1(\Om)$. Then 
\begin{enumerate}[(i)]
    \item $u$ has constant sign.
    \item $u\in C^1(\Om).$
    \item if $v$ is an eigenfunction associated to $\t_1(A_\Om)$, $v$ is radially constant and radially increasing.
\end{enumerate}
\end{proposition}
\begin{proof} Proof follows by similar set of arguments as in the case of euclidean setting: see \cite[proposition A.2]{AnoopAshok} or \cite[Lemma 2.4]{Lindqvist} for $(i)$; for $(ii)$, see \cite[Theorem 1.3]{Barles}; proof of $(iii)$ can be found in \cite[Proposition A.5]{AnoopAshok}. We omit the detailed proof here.
\end{proof}
 Now we give a proof of Theorem \ref{RFK_theorem}. To prove our result, we adapt the ideas used in \cite{Hersch,AnoopAshok,Joy} to the hyperbolic space.
\begin{proof}[Proof of Theorem \ref{RFK_theorem}]
Let $v$ be an eigenfunction of \eqref{Problem} associated to $\t_1(A_\Om).$ Then by Proposition \ref{Properties}, $v$ is radial and it can be chosen positive in $A_\Om$, i.e., $v>0$ and $v(x)=v(d_\H(x,\pa B_r)),$ for all $x\in A_\Om$, where $d_\H$ is the hyperbolic distance function. Let $\M$ be as defined in \eqref{Parame}. Now we represent $v$ in terms of $\M$ in the following way:
$$v(x)=v(d_\H(x,\pa B_r))=(v\circ \M^{-1})(\M(d_\H(x,\pa B_r))),\;\forall\;x\in A_\Om.$$
Let $f=v\circ \M^{-1}.$ Then $v(x)=(f\circ\M)(d_\H(x,\pa B_r)), \;\forall\;x\in A_\Om.$ If $M_*=M(\de_0)$ and $\M_*=\M(R-r)$, then $\M_*\leq M_*.$ Recall that $\Ga_D=\pa \Om_D.$ Now define $u:\Om \longrightarrow \R$ as
\begin{equation*}
 		u(x)= \begin{cases}
 		 (f\circ M)(d_\H(x, \Ga_D)),& \mbox{if } \quad M(d_\H(x, \Ga_D)) \in [0,\M_*], \\
 		f(\M_*), & \mbox{if} \quad M(d_\H(x, \Ga_D)) \in (\M_*,M_*]. 
 		\end{cases}
 		\end{equation*}
 Note that $d_\H(\cdot, \Ga_D)$ is a Lipschitz function. Also $f$ is $C^1$ as $v$ is so (by Proposition \ref{Properties}-$(ii)$). Thus $u\in W^{1,p}(\Om)$. Further, $u(x)=0$ for all $x\in \Ga_D$. Hence $u\in W^{1,p}_{\Ga_D}(\Om).$ Now using the fact that $|\nabla d_\H (x, \Ga_D)|=1, \;\forall\;x\in \Om$ and by the Coarea formula \cite[Theorem 3.1]{Federer}, we get
 \begin{align*}
 			\int_{\Omega}|\nabla u(x)|^p\d V_g &=\int_{\Omega}|\nabla u(x)|^p|\nabla d_\H ( x, \Ga_D)|\d V_g\\ &=\int_{0}^{M^{-1}(\M_*)}\Bigg(\int_{\{x\in \Om:\;d_\H ( x, \Ga_D)=\de\} }|\nabla u(x)|^p\d S\Bigg) \d\de\\
 			&=\int_{0}^{M^{-1}(\M_*)}\Big(|f'(M(\de))|^p|M'(\de)|^p\Big)\Bigg(\int_{\B(\de)}\d S\Bigg) \d\de\\
 			&=\int_{0}^{M^{-1}(\M_*)}\Big(|f'(M(\de))|^p|M'(\de)|^p\Big) L(\de) \d\de\\
 			&=\int_{0}^{M^{-1}(\M_*)}\frac{|f'(M(\delta))|^p}{L(\delta)^{p'-1}}\d\de=\int_{0}^{\M_*}|f'(\be)|^p\d\be,
 		\end{align*}
 where we make a change of variable $M(\delta)=\be$ in the last step. Thus 
 \begin{equation}\label{Grad_u}
     \int_{\Omega}|\nabla u(x)|^p\d V_g=\int_{0}^{\M_*}|f'(\be)|^p\d\be.
 \end{equation}
 Also 
 \begin{align*}
    & \int_{\Omega}| u(x)|^p\d V_g\\ &=\int_{0}^{M^{-1}(M_*)}\Bigg(\int_{\{x\in \Om:\;d_\H ( x, \Ga_D)=\de\} }| u(x)|^p\d S\Bigg) \d\de\\
     &=\int_{0}^{M^{-1}(\M_*)}\Bigg(\int_{\{x\in \Om:\;d_\H ( x, \Ga_D)=\de\} }| u(x)|^p\d S\Bigg) \d\de
     +\int_{M^{-1}(\M_*)}^{\de_0}\Bigg(\int_{\{x\in \Om:\;d_\H ( x, \Ga_D)=\de\} }| u(x)|^p\d S\Bigg) \d\de\\
     &=\int_{0}^{\M_*} f(\be)^p L(M^{-1}(\be))^{p'} \d\be
     +f(\M_*)^p\int_{\M_*}^{M_*} L(M^{-1}(\be))^{p'} \d\be.\quad\qquad[ \text{putting}\;M(\de)=\be]
 \end{align*}
 Therefore, 
 \begin{equation}\label{Norm_u}
     \int_{\Omega}| u(x)|^p\d V_g=\int_{0}^{\M_*} f(\be)^p G(\be)^{p'} \d\be
     +f(\M_*)^p\int_{\M_*}^{M_*} G(\be)^{p'} \d\be,
 \end{equation}
 where $G$ is as defined in Lemma \ref{Auxiliary_lemma}-$(ii)$. By similar methods, we can show that 
 \begin{align}
     \int_{A_\Om}|\nabla v(x)|^p\d V_g &=\int_{0}^{\M_*}|f'(\be)|^p\d\be,\label{Grad_v}\\
     \int_{A_\Om}| v(x)|^p\d V_g &=\int_{0}^{\M_*} f(\be)^p \G(\be)^{p'}\d\be. \label{Norm_v}
 \end{align}
 Observe that, by Proposition \ref{Properties}-$(iii)$, $v$ attains its maxima on $\pa B_R$ and hence we have $f(\be)\leq f(\M_*)$ for all $\be \in [0,\M_*]$. Thus from \eqref{Norm_u}, \eqref{Norm_v} and using Lemma \ref{Auxiliary_lemma}-$(ii)$, we get
 \begin{align*}
     & \int_{A_\Om}| v(x)|^p\d V_g-\int_{\Omega}| u(x)|^p\d V_g\\
     &\leq f(\M_*)^p \int_{0}^{\M_*} \Big( \G(\be)^{p'}-G(\be)^{p'}\Big)\d\be-f(\M_*)^p\int_{\M_*}^{M_*} G(\be)^{p'} \d\be\\
     &= f(\M_*)^p \int_{0}^{\M_*}  \G(\be)^{p'}\d\be-f(\M_*)^p\int_{0}^{M_*} G(\be)^{p'} \d\be\\
     &= f(\M_*)^p \int_{0}^{\M_*} \Big(\L(\M^{-1}(\be))\Big)^{p'}\d\be-f(\M_*)^p \int_{0}^{M_*} \Big(L(M^{-1}(\be))\Big)^{p'}\d\be\\
     &=f(\M_*)^p \int_{0}^{ R-r} \L(\de)\d\de-f(\M_*)^p \int_{0}^{ \de_{0}} L(\de)\d\de=f(\M_*)^p\Big(|A_\Om|-|\Om|\Big).
 \end{align*}
 Since by assumption $|\Om|=|A_\Om|$, we have
 \begin{equation}
     \int_{A_\Om}| v(x)|^p\d V_g\leq \int_{\Omega}| u(x)|^p\d V_g,\label{Norm_compare}
 \end{equation}
 where equality occurs only when $\Om$ is a concentric annular region (by Lemma \ref{Auxiliary_lemma}-$(ii)$). Now the assertion follows substituting \eqref{Grad_u}, \eqref{Grad_v} and \eqref{Norm_compare} in the variational characterization \eqref{Variational} of $\t_1$. Further, equality case immediately comes from the equality case in \eqref{Norm_compare}. This completes the proof.
\end{proof}
\subsection*{Thermal insulation problem:} Let $\Om=\Om_N\setminus \overline{\Om_D}$ be a smooth, doubly connected domain in $\Hn$ as defined in \eqref{Domain}. For $p\in (1,\infty),$ let us consider the following boundary value problem on $\Om$:
\begin{equation}\tag{$\mathscr{T}$}
    \label{Problem2}
	\left. \begin{aligned}
	-\Delta_p u &=0 \qquad\qquad\;\text{in} \quad\Omega,\\
	u &=1 \qquad\quad \quad\;\text{on}\quad \Om_D,\\
|\nabla u|^{p-2}\frac{\partial u}{\partial \eta}+\be |u|^{p-2}u &=0 \qquad\quad \quad\;\text{on} \quad \partial \Om_N,
	\end{aligned}\right\}
\end{equation}
where $\be >0$ is a real parameter and $\eta $ is the outward unit normal to $\partial \Om_N$. Then the energy functional $\mathcal{E}(\Om_D,\Om)$ associated to 
\eqref{Problem2} is given by 
\begin{equation}\label{Energy}
    \mathcal{E}(\Om_D,\Om)=\inf_{v \in W^{1,p}(\Om_N), v \equiv 1 \;\text{in}\; \Om_D}\left\{\int_{\Om}|\nabla v|^p+\be \int_{\partial \Om_N}|v|^p\right\}.
\end{equation}
These types of problems arise in the study of thermal insulation, where a body $\Om_D$ of constant temperature remains surrounded by an insulating material $\Om_N\setminus \overline{\Om_D}$ and $\mathcal{E}(\Om_D,\Om)$ represents the energy of the system; we refer to the book \cite{Bucur2005} for an overview of such problems. Now it is natural to look for the critical configurations of $\Om_D$ and $\Om_N$ so that the energy $\mathcal{E}(\Om_D,\Om)$ is optimized. For planar Euclidean domains, in \cite[Theorem 3.1]{DellaPietra2022}, authors proved that $$\mathcal{E}(\Om_D,\Om_D+\de B_1)\leq \mathcal{E}(\Om_D^\#,\Om_D^\#+\de B_1),$$ where $\de>0$, $\Om_D^\#$ is an open ball with the same perimeter as $\Om_D$, and $B_1$ is the open Euclidean ball of radius one centered at the origin. Here $\Om_D+\de B_1:=\{x+\de y: x\in \Om_D, y\in B_1\}.$ The similar result holds in higher dimensions also if $\Om_D$ is convex and $\Om_D^\#$ is replaced by $\Om_D^*,$ where $\Om_D^*$ is the open Euclidean ball centered at the origin such that $W_{n-1}(\Om_D^*)=W_{n-1}(\Om_D);$ cf. \cite[Theorem 4.1]{DellaPietra2022}. We would like to stress that the hyperbolic analogue of these results can be proved using a similar method developed in this article. To be precise, we can prove the following result.
\begin{theorem}\label{Insulation_Theo}
Let $\Om_D\subset \Hn$ be a smooth, simply connected domain and $\Om_N=\Om_D+\de B_1,$ for some $\de >0,$ i.e., $\Om=(\Om_D+\de B_1)\setminus \overline{\Om_D}.$ Let $\mathcal{E}(\Om_D,\Om)$ be the energy associated to \eqref{Problem2} as defined in \eqref{Energy}. Then the following holds:
\begin{enumerate}[(i)]
    \item if $n=2$ and $\Om_D$ is convex, then $\mathcal{E}(\Om_D,\Om)\leq \mathcal{E}(\Om_D^\#,\Om_D^\#+\de B_1),$ where $\Om_D^\#$ is an open geodesic ball with same perimeter as $\Om_D.$
    
    \item if $n\geq 2$ and $\Om_D$ is h-convex, then $\mathcal{E}(\Om_D,\Om)\leq \mathcal{E}(\Om_D^*,\Om_D^*+\de B_1),$ where $\Om_D^*$ is an open geodesic ball with same $(n-1)$-th quermassintegral as $\Om_D.$
\end{enumerate}
Moreover, equality occurs in both the above cases when $\Om_D$ is an open geodesic ball in $\Hn$.
\end{theorem}

%% file: Open_problem.tex
\section{Final comments and open problems}\label{open_problem}
\begin{remark} We conclude this article by introducing some immediate open questions.
\begin{enumerate}[(i)]
    \item As pointed out in Remark \ref{Perimeter constraint}, for $n\geq 3,$ $K^*$ has a larger perimeter than $K$. It is not known whether Theorem \ref{Nagy_theorem}-$(ii)$ will hold or not if $K^*$ is replaced by $K^\#,$ where $K^\#$ is an open geodesic ball such that $P(K)=P(K^\#).$
    \item For $n=2,$ we have established the reverse Faber-Krahn inequality (Theorem \ref{RFK_theorem}) under the assumption that $\Om_D$ is a convex (geodesically) domain in $\Hn$. This assumption is necessary to apply the hyperbolic Steiner formula \eqref{Steiner_formula}. Therefore, our approach of proofs is not applicable if $\Om_D$ is not convex. Of course, it could be an interesting problem to study when $\Om_D$ is a non-convex domain. However, this seems to be a challenging problem at this moment.
    
    \item For $n\geq 3,$ we proved Theorem \ref{RFK_theorem} when $\Om_D$ is a $h$-convex domain in $\Hn$. Such assumption is essential in order to use the hyperbolic Alexandrov-Fenchel inequality (Theorem \ref{Alexandrov}) which is a crucial tool in proving Nagy's type inequality (Theorem \ref{Nagy_theorem}-$(ii)$). To the best of our knowledge, a similar version of Alexandrov-Fenchel inequality is not available in $\Hn$ if the domain is not $h$-convex.
\end{enumerate}

\end{remark}

%% file: Acknowledgements.tex
\section{Acknowledgements}
The authors are grateful to Prof. T. V. Anoop for his valuable comments and many fruitful discussions about the results of this article.